%% file: NegativeRates.tex
\documentclass[a4paper,10pt]{scrartcl}

\usepackage[fleqn]{amsmath}
\usepackage{amssymb}
\usepackage{amsthm}
\usepackage{stmaryrd}
\usepackage{bbm}
\usepackage{hyperref}
\usepackage{mathtools}
\usepackage{xcolor}

\mathtoolsset{showonlyrefs}

\include{Include}

\include{Theorems}

\newcommand{\hX}{\widehat{X}}
\newcommand{\hV}{\widehat{V}}

\begin{document}

\title{Markov process representation of semigroups whose generators include negative rates}
\author{F. V\"ollering\footnote{florian.voellering@math.uni-leipzig.de, University of Leipzig}}

\maketitle
\begin{abstract}
Generators of Markov processes on a countable state space can be represented as finite or infinite matrices. One key property is that the off-diagonal entries corresponding to jump rates of the Markov process are non-negative. Here we present stochastic characterizations of the semigroup generated by a generator with possibly negative rates. This is done by considering a larger state space with one or more particles and antiparticles, with antiparticles being particles carrying a negative sign.
\end{abstract}

\section{Introduction}

Consider the generator $L$ of a Markov jump process $(X_t)_{t\geq 0}$ on a countable state space $E$. It is characterized by jump rates $r(x,y)$ for jumps from $x$ to $y, x\neq y$, $r(x,x)=0$, and for $f:E\to \bR$
\begin{align}\label{eq:generator}
Lf(x) = \sum_{y\in E}r(x,y)[f(y)-f(x)].
\end{align}
The relationship between the probabilistic process $(X_t)_{t\geq0}$, its semi-group $(P_t)_{t\geq0}$ with $P_tf(x)=\bE_x f(X_t)$ and generator describing the rules for jumps is very fruitful. One essential restriction is that the jump rates are non-negative. If $r(x,y)<0$ is allowed, then \eqref{eq:generator} is still a perfectly valid operator which under reasonable conditions will be the generator of a semi-group $S_t=e^{tL}$, but the probabilistic interpretation is lost. The aim of this note is to recover some probabilistic meaning.

Before we go into the details let us remind us of some basic facts. 
In the probabilistic setting, the generator $L$ is usually characterized via its jump rates $r(x,y)$. If we consider $L$ as matrix, then its off-diagonal entries are given by $r(x,y)$, while the diagonal is given by $-\sum_{y:x\neq y} r(x,y)$. 
The fact that a Markov generator as a matrix has zero sum rows stems form the preservation of mass. The off diagonal entry $r(x,y)$ is the parameter of the exponential waiting time for a jump from $x$ to $y$. When presented with a matrix $A$ where the diagonal entries do not match $-\sum_{y:x\neq y} r(x,y)$ but the off-diagonal entries are non-negative, then the deviation can be split off into a potential $V$, writing 
\begin{align}\label{eq:A}
    Af(x)=(L+V)f(x) = \sum_{y\in E}r(x,y)[f(y)-f(x)] + V(x)f(x)
\end{align}
with $V$ a diagonal matrix and $L$ a Markov generator. The potential term $V(x)$ has the probabilistic interpretation of a branching or killing rate, and the corresponding semi-group has (assuming for simplicity finite row sum norm of $A$)  the explicit probabilistic form
\begin{align}\label{eq:FeynmanKac}
    e^{tA}f(x) = \bE_x f(X_t)e^{\int_0^t V(X_u)du}
\end{align}
with $X_t$ the Markov process generated by $L$. Equation \eqref{eq:FeynmanKac}
 is sometimes referred to as the Feynman-Kac formula. The basic intuition behind this formula is that it represents the expectation of particles moving independently according to $L$, and which are killed or branch into two at rate $|V(x)|$, with negative $V$ implying killing. 

We will build on this intuition to deal with negative jump rates, which should represent both movement and killing. We can consider a regular jump event from $x$ to $y$ via positive rates $r(x,y)>0$ as the killing of a particle at $x$ and creation of a particle at $y$. Correspondingly we will see that a 'jump' event from negative rates $r(x,y)<0$ is in some sense the opposite, the destruction of a particle at $y$ and the creation of one at $x$. This runs into the problem that there might be no particle at $y$ to destroy. We solve this by introducing anti-particles, and consider killing a particle at $y$ the same as creating an anti-particle at $y$. In the following sections we will look at the details, with Theorem \ref{thm:particle-antiparticle} corresponding to a single (anti-)particle like in \eqref{eq:FeynmanKac} and Theorems \ref{thm:branching} and Theorem \ref{thm:annihilation}
giving multi-particle formulations. Section \ref{sec:duality} looks at an application to duality of Markov processes and Section \ref{sec:example} gives the simple example of a double Laplacian.

\section{Switching between particles and antiparticles}
Let us write $r^+(x,y)=\max(r(x,y),0)$ and $r^-(x,y)=\max(-r(x,y),0)$, and consider the Markov process $(\hX_t, Z_t)_{t\geq0}$ on $E\times\{-1,+1\}$ with generator
\begin{align}
\widehat{L}f(x,s) = \sum_{y\in E}r^+(x,y)\left[f(y,s)-f(x,s)\right]+\sum_{y\in E}r^-(x,y)\left[f(y,-s)-f(x,s)\right].
\end{align}
We interpret $\hX_t$ as the position of the Markov process, and $Z_t$ indicates whether it is a particle ($Z_t=+1$) or an anti-particle ($Z_t=-1$). Then the first sum describes just the regular change of position via jumps utilizing the rates $r^+$. The second sum similarly describes movement, but whenever the particle jumps according to the rates $r^-$, the state also changes from particle to anti-particle or vice versa. We can now present a stochastic representation of the semi-group generated by an arbitrary matrix with finite supremum norm.

\begin{theorem}\label{thm:particle-antiparticle}
Let $A$ be of the form 
\[ Af(X) = \sum_{y\in E}r(x,y)[f(y)-f(x)] +V(x)f(x) \]
and assume $\sup_{x\in E}\sum_{y\in E} |r(x,y)| <\infty$, $\sup_{x\in E}|V(x)| <\infty$, $r(x,x)=0$. Then $A$ is a bounded operator w.r.t. the supremum-norm, $S_t=e^{tA}$ is well-defined and for any $f:E\to\bR$ bounded, we have
\begin{align}\label{eq:particle-antiparticle-solution}
S_tf(x) = \bE_{x,+1}\left[Z_t f(\hX_t)e^{2\int_0^t \sum_{y\in E}r^-(\hX_u,y)+V(\hX_u)du }\right].
\end{align}
\end{theorem}
\begin{proof}
Write $\hat{f}(x,s)=sf(x)$ and $\hV(x,s) = 2\sum_{y\in E}r^-(x,y)+V(x)$. Then the right hand side of \eqref{eq:particle-antiparticle-solution} is the Feynman-Kac formulation of the solution of
\begin{align}
\begin{cases}
\frac{\partial \phi_t}{\partial t}(x,s) = \widehat{L}\phi_t(x,s) + \hV(x,s)\phi_t(x,s),	\\
\phi_0 = \hat{f}.
\end{cases}
\end{align}
On the other hand, $\tilde\phi_t(x,s):=s S_t f(x)$ also satisfies
\begin{align*}
\frac{\partial \tilde\phi_t}{\partial t}(x,s) = s A S_t f(x) = \widehat{L}\tilde\phi_t(x,s) + \hV(x,s)\tilde\phi_t(x,s),
\end{align*}
and since $\tilde\phi_0 = \phi_0$ the claim \eqref{eq:particle-antiparticle-solution} follows.
\end{proof}

\section{Branching particles and antiparticles}\label{sec:branching}
Consider a system of particles $\eta_t^+\in \bN_0^E$ and antiparticles $\eta_t^-\in \bN_0^E$, where $\eta^\pm_t(x)$ is the number of particles/anti-particles at site $x$ and time $t$. These particles move independently with jump rates $r^+(x,y)$. Additionally there is the following branching mechanism: a particle at site $x$ branches into two particles at $x$ and one anti-particle at site $y$ at rate $r^-(x,y)$. The same is true for antiparticles at $x$, which branch into two at $x$ plus a particle at $y$. 
The generator describing the movement and branching of particles is
\begin{align}
L_+^\uparrow f (\eta^+,\eta^-) 
&= \sum_{x,y}r^+(x,y)\eta^+(x)[f(\eta^+ + \delta_y-\delta_x, \eta^-) - f(\eta^+,\eta^-)]	\\
&+ \sum_{x,y}r^-(x,y)\eta^+(x)[f(\eta^+ + \delta_x, \eta^- +\delta_y) - f(\eta^+,\eta^-)]	.
\end{align}
The first line of the generator describes the movement of particles. The rate $r^+(x,y)\eta^+(x)$ is the total rate that one of the $\eta^+(x)$ many particles at $x$ jumps from $x$ to $y$. After this jump there is one less particle at $x$ and one more at $y$, making the new particle configuration $\eta^++\delta_y-\delta_x$. The configuration of anti-particles $\eta^-$ is unchanged. 
The second line describes the branching mechanism, with $r^-(x,y)\eta^+(x)$ the aggregate rate that one of the particles at $x$ turns into two particles at $x$ and one anti-particle at $y$. In total the result is one more particle at $x$ and one more anti-particle at $y$, resulting in the change $(\eta^+,\eta^-) \rightarrow (\eta^++\delta_x, \eta^-+\delta_y)$.

The generator describing the movement and branching of anti-particles is analogous, with the roles of particles and anti-particles reversed:
\begin{align}
L_-^\uparrow f (\eta^+,\eta^-) 
&= \sum_{x,y}r^+(x,y)\eta^-(x)[f(\eta^+, \eta^- + \delta_y-\delta_x) - f(\eta^+,\eta^-)]	\\
&+ \sum_{x,y}r^-(x,y)\eta^-(x)[f(\eta^+ + \delta_y, \eta^- +\delta_x) - f(\eta^+,\eta^-)]	.
\end{align}
The generator $L^\uparrow = L^\uparrow_+ + L^\uparrow_-$  then describes the total system. This system is well-defined under the assumption that $\sup_{x\in E}\sum_{y\in E} |r(x,y)|=M <\infty$, which guarantees that there is no explosion: if $N_t=\sum_x \eta_t^+(x) +\sum_x \eta_t^-(x)$ is the total number of particles and anti-particles in the system, then $N_t$ is dominated by a jump process with jumps from $n$ to $n+2$ at rate $n M$, which leads to exponential growth but no explosion.
Also note that under the dynamics the number $\sum_x \eta_t^+(x) - \sum_x \eta_t^-(x)$ is preserved in time. 
In particular, for the system starting with a single particle at $x$, i.e., $\eta^+_0=\delta_x$ and $\eta^-_0=0$, the sum is always 1.

\begin{theorem}\label{thm:branching}
Assume $\sup_{x\in E}\sum_{y\in E} |r(x,y)| <\infty$. Given $f:E\to\bR$ bounded, define 
\begin{align}
f^\uparrow(\eta^+,\eta^-) = \sum_{x \in E} (\eta^+(x)-\eta^-(x))f(x).
\end{align}
Then the semigroup $S_t$ generated by \eqref{eq:generator} has the stochastic description
\begin{align}\label{eq:branching-solution}
S_tf(x) = \bE_{(\delta_x,0)} f^\uparrow(\eta^+_t,\eta^-_t) .
\end{align}
\end{theorem}
\begin{proof}
Let $(\eta^{+,i}_t,\eta^{-,i}_t)_{t\geq 0}$, $i=1,...,n$ be independent realizations of the particle system started at $(\eta^{+,i}_0,\eta^{-,i}_0)$. Then, by the independence of the branching and movement of particles, $(\sum_{i=1}^n \eta^{+,i}_t, \sum_{i=1}^n \eta^{-,i}_t)_{t\geq 0}$ has the same law as a a system started in $(\sum_{i=1}^n \eta^{+,i}_0, \sum_{i=1}^n \eta^{-,i}_0)$. As a consequence, since $f^\uparrow$ is linear in $\eta^+,\eta^-$, and anti-symmetric under exchange of $\eta^+$ and $\eta^-$,
\begin{align}
\bE_{\eta^+_0,\eta^-_0} f^\uparrow(\eta^+_t,\eta^-_t) 
&= \sum_{x}\eta_0^+(x)\bE_{\delta_x,0} f^\uparrow(\eta^+_t,\eta^-_t) + \sum_{x}\eta_0^-(x)\bE_{0,\delta_x} f^\uparrow(\eta^+_t,\eta^-_t)	\\
&= \sum_{x}\eta_0^+(x)\bE_{\delta_x,0} f^\uparrow(\eta^+_t,\eta^-_t) - \sum_{x}\eta_0^-(x)\bE_{\delta_x,0} f^\uparrow(\eta^+_t,\eta^-_t).	\label{eq:lift-linear}
\end{align}
and in particular
\begin{align}
\bE_{2\delta_x,\delta_y} f^{\uparrow}(\eta^+_t,\eta^-_t) - \bE_{\delta_x,0} f^{\uparrow}(\eta^+_t,\eta^-_t) = \bE_{\delta_x,0} f^{\uparrow}(\eta^+_t,\eta^-_t) - \bE_{\delta_y,0} f^{\uparrow}(\eta^+_t,\eta^-_t).
\end{align}
If we write $u_t(x)=\bE_{\delta_x,0}f^{\uparrow}(\eta^+_t,\eta^-_t)$, then 
\begin{align}
\frac{d}{dt} u_t(x) 
&= \left[L^\uparrow \bE_{\cdot}f^{\uparrow}(\eta^+_t,\eta^-_t)\right](\delta_x,0)		\\
&= \sum_{y} r^+(x,y)\left[ \bE_{\delta_y,0} f^{\uparrow}(\eta^+_t,\eta^-_t) - \bE_{\delta_x,0} f^{\uparrow}(\eta^+_t,\eta^-_t) \right]	\\
&\quad+ \sum_{y} r^-(x,y)\left[ \bE_{2\delta_x,\delta_y} f^{\uparrow}(\eta^+_t,\eta^-_t) - \bE_{\delta_x,0} f^{\uparrow}(\eta^+_t,\eta^-_t) \right]	\\
&= \sum_{y} r^+(x,y)\left[ \bE_{\delta_y,0} f^{\uparrow}(\eta^+_t,\eta^-_t) - \bE_{\delta_x,0} f^{\uparrow}(\eta^+_t,\eta^-_t) \right]	\\
&\quad+ \sum_{y} r^-(x,y)\left[ \bE_{\delta_x,0} f^{\uparrow}(\eta^+_t,\eta^-_t) - \bE_{\delta_y,0} f^{\uparrow}(\eta^+_t,\eta^-_t) \right]		\\
&= L u_t(x).
\end{align}
Hence $u_t(x)$ is the unique solution of 
\begin{align}
\begin{cases}
\frac{\partial u_t}{\partial t}(x) = L u_t(x),	\\
u_0 = f(x).
\end{cases}
\end{align}
\end{proof}
\begin{remark}
Theorem \ref{thm:branching} assumes for simplicity and readability that there is no potential. The presence of a potential $V$ like in Theorem \ref{thm:particle-antiparticle} would mean that there is in addition branching and annihilation of particles and antiparticles via
\begin{align}
L^\uparrow_V f(\eta^+,\eta^-) &= \sum_x V^+(x)\eta^+(x) [f(\eta^++\delta_x,\eta^-) -f(\eta^+,\eta^-)] \\
	&+ \sum_x V^+(x)\eta^-(x) [f(\eta^+,\eta^-+\delta_x) -f(\eta^+,\eta^-)] \\
	&+\sum_x V^-(x)\eta^+(x) [f(\eta^+-\delta_x,\eta^-) -f(\eta^+,\eta^-)] \\
	&+ \sum_x V^-(x)\eta^-(x) [f(\eta^+,\eta^--\delta_x) -f(\eta^+,\eta^-)].
\end{align}
meaning both particles and anti-particles individually branch into two at rate $V^+$ or are killed at rate $V^-$.
It can be easily verified that $[L^\uparrow_V \bE_\cdot f^\uparrow(\eta^+_t,\eta^-_t)](\delta_x,0) = V(x)u_t(x)$.
\end{remark}

\section{Branching and annihilating particles and antiparticles}\label{sec:annihilating}
The process in Section \ref{sec:branching} tends to have an exponentially growing number of particles. It turns out that we can introduce annihilation of particles and antiparticles to reduce this number. We do so by letting any pair of particle and antiparticle which are at the same site annihilate at rate $\lambda\in[0,\infty]$, where infinite rate corresponds to instant annihilation. Let
\begin{align}
L^{\uparrow, \lambda}f(\eta^+,\eta^-) = L^\uparrow f(\eta^+,\eta^-) + \lambda\sum_{x}\eta^+(x)\eta^-(x)\left[f(\eta^+-\delta_x,\eta^--\delta_x) - f(\eta^+, \eta^-)\right]
\end{align}
be the generator of the particle system which includes annihilation.
\begin{theorem}\label{thm:annihilation}
Theorem \ref{thm:branching} is also valid when there is annihilation for any $\lambda\in(0,\infty]$.
\end{theorem}
\begin{proof}
Write $P_t^{\uparrow,\lambda}f(\eta^+,\eta^-)=\bE_{\eta^+,\eta^-}f(\eta^+_t,\eta^-_t)$ for the semigroup generated by $L^{\uparrow,\lambda}$, with $\lambda=0$ being the system without annihilation. By \eqref{eq:lift-linear}, if $\eta^+(x)>0$ and $\eta^-(x)>0$, 
\[ P_t^{\uparrow,0} f^\uparrow(\eta^+,\eta^-) = P_t^{\uparrow,0} f^\uparrow(\eta^+-\delta_x,\eta^--\delta_x).	\]
Hence
\begin{align}
(L^{\uparrow,\lambda}-L^{\uparrow,0}) P_t^{\uparrow,0} f^\uparrow(\eta^+,\eta^-) = 0
\end{align}
and it follows that $P_t^{\uparrow,\lambda}f^{\uparrow} = P_t^{\uparrow,0}f^{\uparrow}$.
\end{proof}

\section{Applications to duality of Markov processes}\label{sec:duality}
A very brief introduction to duality of Markov processes is as follows. Two Markov processes $(X_t)_{t\geq0}$ and $(Y_t)_{t\geq0}$ on state spaces $E$ and $F$ are said to be dual with duality function $H:E\times F\to \bR$, if for all $x\in E$ and $y\in F$,
\begin{align}\label{eq:duality-def}
\bE_x H(X_t;y) = \bE_y H(x ;Y_t).
\end{align}
A sufficient condition is that the generators $L_X$ and $L_Y$ satisfy
\begin{align}\label{eq:duality-generators}
[L_X H(\cdot; y)](x) = [L_Y H(x; \cdot)](y), \quad\forall\;x\in E, y\in F.
\end{align}
Duality has proven fruitful in many applications. For a survey on duality, see \cite{JANSEN:KURT:14}. The challenge with duality is that given a Markov process $X_t$ of interest, how to find a Markov process $Y_t$ and duality function $H$ so that \eqref{eq:duality-def} holds. One can make an educated guess on $H$, and then find a generator $L_Y$ which satisfies \eqref{eq:duality-generators}. Or one can use symmetries of $L_X$ to identify a suitable Lie algebra representation whose building blocks can build $L_X$, and then find a dual representation, which then allows to build $L_Y$, see \cite{GIARDINA:KURCHAN:REDIG:VAFAYI:2009} and  \cite{STURM:SWART:VOELLERING:2018} for an introduction to this method. However, neither method guarantees that the dual generator $L_Y$ is actually a Markov generator. If $F$ is countable, as is the case in many applications of duality, then $L_Y$ can be represented as a finite or infinite matrix. A stochastic representation of the semigroup generated by such an $L_Y$ is desirable, and with Theorem \ref{thm:particle-antiparticle}, Theorem \ref{thm:branching} or Theorem \ref{thm:annihilation} this is possible. 
\begin{theorem}
Assume that there is a duality function $H$ and generator $L_Y$ satisfying \eqref{eq:duality-generators}, with $F$ countable. Further assume that the matrix representation of $L_Y$ has row sums 0, so that it can be written in the form of \eqref{eq:generator}, and $\sup_{y\in F}\sum_{z\in F}|r(y,z)|<\infty$. Then the Markov process $(X_t)_{t\geq 0}$ is dual to the process $(\eta_t^+,\eta_t^-)_{t\geq 0}$ with duality function 
\[ H^\uparrow(x;\eta^+,\eta^-) = \sum_{y\in F} (\eta^+(y)-\eta^-(y))H(x;y).	\]
Here $(\eta_t^+,\eta_t^-)_{t\geq 0}$ is the branching (and annihilating) particle system introduced in sections \ref{sec:branching} and \ref{sec:annihilating}, with arbitrary annihilation rate $\lambda\in[0,\infty]$. In other words
\begin{align}\label{eq:duality}
\bE_x H^\uparrow(X_t;(\eta^+,\eta^-)) = \bE_{\eta^+,\eta^-} H^\uparrow(x;(\eta_t^+,\eta^-_t)).
\end{align}
\end{theorem}
\begin{proof}
By the proof of Theorem \ref{thm:annihilation} the right hand side of \eqref{eq:duality} does not depend on the annihilation rate, so we can restrict ourself to the case of no annihilation.
By \eqref{eq:duality-generators} we have $\bE_x H(X_t;y) = [S_t H(x ;\cdot)](y)$, where $S_t$ is the semigroup generated by $L_Y$. Then, by Theorem \ref{thm:branching}, we have 
\begin{align}
\bE_x H^\uparrow(X_t;(\delta_y,0)) = \bE_x H(X_t;y) = [S_t H(x ;\cdot)](y) = \bE_{\delta_y,0} H^\uparrow(x;(\eta_t^+,\eta^-_t)).
\end{align}
Finally, with \eqref{eq:lift-linear} we can extend the above from $(\delta_y,0)$ to arbitrary starting configurations.
\end{proof}

\section{Example: Double Laplacian on the integers}\label{sec:example}
Let $\Delta f(x) = \frac12f(x+1)-f(x)+\frac12f(x-1)$ be the discrete Laplacian on $\bZ$. Then the double Laplacian is given by
\begin{align}
\Delta \Delta f(x) &= \frac14(f(x+2)-f(x)) + \frac14(f(x-2)-f(x)) \\
& \quad - (f(x+1)-f(x)) - (f(x-1)-f(x)),
\end{align}
which is of the form \eqref{eq:generator} with negative rates. Let $S_t$ be the semigroup generated by the double Laplacian $\Delta\Delta$. We will apply Theorem \ref{thm:particle-antiparticle}. So let $\hX$ be the random walk on $\bZ$ which performs the jumps $\pm1$ at rate 1 and $\pm2$ at rate $\frac14$. Since jumps using the rates $r^{-}$ involve flipping the sign of $Z_t$, we have that $Z_t=(-1)^{N_t}$, where $N_t$ is the number of nearest neighbour jumps performed by $\hX_t$. Note that $N_t$ is even iff $\hX_t-\hX_0$ is even. Hence 
\begin{align}\label{eq:double-laplace-Z_t}
Z_t=2\ind_{N_t\text{ is even}} -1 = 2\ind_{\hX_t-\hX_0\text{ is even}}-1.
\end{align}
Finally we observe that by spatial homogeneity $\sum_{y}r^-(x,y) = 2$. 
By Theorem \ref{thm:particle-antiparticle},
\begin{align}\label{eq:double-laplace-solution}
S_t f(x) = e^{4t}\bE_{x}\left(Z_t f(\hX_t)\right).
\end{align}
Note that $N_t$ is Poisson($2t$)-distributed, and therefore $\bP(N_t \text{ is even})=\frac12(1+e^{-4t})$ and $\bE Z_t = e^{-4t}$. Alternatively, $\bE Z_t = e^{-4t}$ follows from \eqref{eq:double-laplace-solution} applied to the constant function $\mathbf{1}$, since $S_t \mathbf{1} =\mathbf{1}$.
For a more complex example consider $f$ of the form $f(x)=g(x)\ind_{x\text{ is even}}$. Then, by \eqref{eq:double-laplace-Z_t} and \eqref{eq:double-laplace-solution},
\begin{align}
S_t f(x) = \begin{cases}
\frac12(e^{4t}+1)\bE_x\big[g(\hX_t)\big|\hX_t \text{ even}\big],	\quad&x \text{ even};	\\
-\frac12(e^{4t}-1)\bE_x\big[g(\hX_t)\big|\hX_t \text{ odd}\big],\quad&x \text{ odd}.
\end{cases}
\end{align}
The conditional expectations are reasonably well approximated by integrating $g$ against a normal distribution with variance $\Var(\hX_t)=4t$ assuming $g$ is smooth enough and $t$ not too small.

\section{Example of all rates negative}
Consider the operator of the form \eqref{eq:A} with all $r(x,y)\leq 0$. We make the simplifying assumption that there are constants $\lambda_1,\lambda_2$ so that $\sum_y r^-(x,y) = \lambda_1$ and $V(x)=\lambda_2$ for all $x$. Then, by Theorem \ref{thm:particle-antiparticle},
\begin{align}
&e^{-(2\lambda_1+\lambda_2)t} e^{At}f(x) = \bE_{x,+1} \left( Z_t f(\hX_t) \right) \label{eq:decay-example1}\\
&= \bE_{x,+1} \left[ f(\hX_t)\;\middle|\;N_t \text{ even}\right] \bP(N_t \text{ even}) - \bE_{x,+1} \left[ f(\hX_t)\;\middle|\;N_t \text{ odd}\right] \bP(N_t \text{ odd}),
\end{align}
where $N_t$ counts the number of jumps of $\hX_t$. Since $\sum_y r^-(x,y) = \lambda_1$ it follows that $N_t$ is Poisson$(\lambda_1 t)$-distributed and $\bP(N_t \text{ even})=\frac12(1+e^{2\lambda_1 t})$. 

If we assume that $\hX_t$ has a stationary distribution $\mu$  then it is reasonable to write 
\begin{align}
 \bE_{x,+1} \left[ f(\hX_t)\;\middle|\;N_t \text{ even}\right] &= \mu(f) + b_t^e(x);	\\
 \bE_{x,+1} \left[ f(\hX_t)\;\middle|\;N_t \text{ odd}\right] &= \mu(f) + b_t^o(x).
\end{align}
If $\hX_t$ is converging exponentially fast to $\mu$ the error terms $b_t^e(x)$ and $b_t^o(x)$ will be decaying exponentially at some rate $0\leq\nu\leq \lambda_1$ (if the Markov process is on a bipartite graph $\nu$ can be 0 even if convergence to $\mu$ is exponentially fast, and the rate is no larger than $\lambda_1$ since that the exit rate for a single site). Then, continuing from \eqref{eq:decay-example1},
\begin{align}\label{eq:decay-example2}
&e^{At}f(x) = e^{(2\lambda_1+\lambda_2)t}\frac{b_t^e(x)-b_t^o(x)}2 +e^{\lambda_2t}\left(\mu(f) + \frac{b_t^e(x)+b_t^o(x)}2\right).
\end{align}
 Therefore typically the first term is the dominant term, and $|e^{At}f|$ grows at rate at most $2\lambda_1+\lambda_2-\nu$, which is slower than what \eqref{eq:decay-example1} or the supremum norm $\norm{A}_\infty=|\lambda_1+\lambda_2|+\lambda_1$ suggest. More details depend on a more sophisticated analysis using specifics of $f$ and $\hX_t$, for example by finding cancellations in $b_t^e(x)-b_t^o(x)$.

\bibliography{BibCollection}{}
\bibliographystyle{amsplain}

\end{document}

%% file: Include.tex
\DeclareMathOperator{\Var}{Var}

\newcommand{\bE}{\ensuremath{\mathbb{E}}}

\newcommand{\bN}{\ensuremath{\mathbb{N}}}

\newcommand{\bP}{\ensuremath{\mathbb{P}}}

\newcommand{\bR}{\ensuremath{\mathbb{R}}}

\newcommand{\bZ}{\ensuremath{\mathbb{Z}}}

\newcommand{\ind}{\ensuremath{\mathbbm{1}}}

\newcommand{\norm}[1]{\left\Vert \, #1 \, \right\Vert}

\newcommand{\ddx}[1][1]{\ifnum#1=1 \frac{d}{dx} \else \frac{d^{#1}}{dx^{#1}} \fi}
\newcommand{\ddy}[1][1]{\ifnum#1=1 \frac{d}{dy} \else \frac{d^{#1}}{dy^{#1}} \fi}
\newcommand{\ddt}[1][1]{\ifnum#1=1 \frac{d}{dt} \else \frac{d^{#1}}{dt^{#1}} \fi}



%% file: Theorems.tex
\newtheorem{theorem}{Theorem}[section]

\newenvironment{remark}[1][Remark]{\begin{trivlist}
\item[\hskip \labelsep {\bfseries #1}]}{\end{trivlist}}